\documentclass[11pt,reqno]{amsart}

\usepackage[margin=3cm]{geometry}

\usepackage{amsmath,amssymb,amsthm}
\usepackage{verbatim}

\usepackage{xurl}
\usepackage[psdextra]{hyperref}
\usepackage[capitalize]{cleveref}
\IfFormatAtLeastTF{2026-06-01}{%
  \newcommand*{\preamblesharedtheorem}[2]{%
    \newtheorem{#1}[theorem]{#2}%
  }%
}{%
  \usepackage{aliascnt}
  \newcommand*{\preamblesharedtheorem}[2]{%
    \newaliascnt{#1}{theorem}%
    \newtheorem{#1}[#1]{#2}%
    \aliascntresetthe{#1}%
  }%
}

\usepackage{enumerate}
\usepackage{enumitem}
\usepackage{physics}
\usepackage{todonotes}
\definecolor{crimson}{HTML}{C8102E}
\definecolor{stopgreen}{HTML}{00A14B}
\presetkeys{todonotes}{inline,color=yellow}{}

\usepackage[normalem]{ulem}

\usepackage[OT2,T1]{fontenc}
\usepackage{microtype}
\DeclareSymbolFont{cyrletters}{OT2}{wncyr}{m}{n}
\DeclareMathSymbol{\sha}{\mathalpha}{cyrletters}{"58}

\usepackage{tikz}
\usetikzlibrary{cd,decorations.markings}

\usepackage{graphicx}

\usepackage{outlines}
\usepackage{stmaryrd}
\usepackage{mathrsfs}
\usepackage{expl3}

\AddToHook{begindocument/before}{%
  \IfPackageLoadedTF{biblatex}{%
    \DeclareFieldFormat[online]{title}{\mkbibquote{#1\isdot}}%
  }{}%
}

\makeatletter
\DeclareFontFamily{OMX}{MnSymbolE}{}
\DeclareSymbolFont{MnLargeSymbols}{OMX}{MnSymbolE}{m}{n}
\SetSymbolFont{MnLargeSymbols}{bold}{OMX}{MnSymbolE}{b}{n}
\DeclareFontShape{OMX}{MnSymbolE}{m}{n}{
    <-6>  MnSymbolE5
   <6-7>  MnSymbolE6
   <7-8>  MnSymbolE7
   <8-9>  MnSymbolE8
   <9-10> MnSymbolE9
  <10-12> MnSymbolE10
  <12->   MnSymbolE12
}{}
\DeclareFontShape{OMX}{MnSymbolE}{b}{n}{
    <-6>  MnSymbolE-Bold5
   <6-7>  MnSymbolE-Bold6
   <7-8>  MnSymbolE-Bold7
   <8-9>  MnSymbolE-Bold8
   <9-10> MnSymbolE-Bold9
  <10-12> MnSymbolE-Bold10
  <12->   MnSymbolE-Bold12
}{}

\let\llangle\@undefined
\let\rrangle\@undefined
\DeclareMathDelimiter{\llangle}{\mathopen}{MnLargeSymbols}{'164}{MnLargeSymbols}{'164}
\DeclareMathDelimiter{\rrangle}{\mathclose}{MnLargeSymbols}{'171}{MnLargeSymbols}{'171}
\makeatother

\renewcommand{\phi}{\varphi}

\newcommand{\ovl}{\overline}

\let\emptyset\varnothing

\ExplSyntaxOn
\cs_new_protected:Npn \preamble_new_operator:nn #1#2
  { \cs_new:cpn {#1} { \operatorname{#2} } }
\cs_new_protected:Npn \preamble_set_operator:nn #1#2
  { \cs_set:cpn {#1} { \operatorname{#2} } }
\cs_new_protected:Npn \preamble_new_font_alias:nnn #1#2#3
  { \cs_new:cpn {#1#2} { #3 {#2} } }

\clist_map_inline:nn
  {
    ord,Gal,Spec,Hom,SL,SO,GL,Isom,Ind,Cl,Pic,End,Sym,Sel,Tor,codim,Br,Th,Lind,Id,id,
    PSL,Aut,PGL,pr,Rad,Proj,depth,Ext,gr,nil,red,rad,Vect,MT,Hdg,Sch,Frob,Ver,Fr,fr,
    Vol,Div,Supp,cl,Sw,Coim,Gr,Ker,Coker,coker,Cov,Mod,Lie,Stab,acl,dcl,Mult,eq,ACF,
    Kr,Ded,Av,PD,st,Ch,Var,Swan,val,rk,Der,SU,SH,Ann,Cone,sel,tp,stp,supp,tor,gp,
    disc,Cn,Jac,Vis,WC,Comp,Frac,Func,Perv,Sp,diag,Nrd,an,Hilb,Cliff,Ad,Diff,sing,
    Out,Bl,rec,Inn,Loc,Coh,Aff,colim,coim,Ob,Sub,Def,alg,nr,un,ad,Fl,ab,fl,Ric,loc,
    sm,Mor,aut,Hol,An,Arg,BGL,BU,BO,full,om,Sing,St,Lk,ch,td,per,ind
  }
  { \preamble_new_operator:nn {#1} {#1} }

\preamble_new_operator:nn {Char} {char}
\preamble_new_operator:nn {Maxspec} {MaxSpec}
\preamble_new_operator:nn {ext} {Ext}
\preamble_new_operator:nn {height} {ht}
\preamble_new_operator:nn {vol} {Vol}
\preamble_new_operator:nn {rM} {r-M}
\preamble_new_operator:nn {im} {im}
\preamble_new_operator:nn {Sh} {SH}
\preamble_new_operator:nn {cone} {Cone}
\preamble_new_operator:nn {defect} {def}
\preamble_new_operator:nn {resid} {res}
\preamble_new_operator:nn {Graff} {GrAff}

\preamble_set_operator:nn {Form} {Form}
\preamble_set_operator:nn {div} {div}
\preamble_set_operator:nn {ker} {Ker}
\preamble_set_operator:nn {Im} {Im}
\preamble_set_operator:nn {df}{def}
\ExplSyntaxOff

\newcommand{\SHom}{\mathcal{H}\kern -.5pt\mathit{om}}
\newcommand{\shom}{\mathcal{H}\kern -.5pt\mathit{om}}
\newcommand{\sext}{\mathcal{E}\kern -.5pt\mathit{xt}}

\renewcommand{\subset}{\subseteq}

\ExplSyntaxOn
\clist_map_inline:nn {A,B,C,D,E,F,G,H,I,J,K,L,M,N,O,P,Q,R,S,T,U,V,W,X,Y,Z}
  {
    \preamble_new_font_alias:nnn {b}  {#1} {\mathbf}
    \preamble_new_font_alias:nnn {bb} {#1} {\mathbb}
    \preamble_new_font_alias:nnn {s}  {#1} {\mathscr}
    \preamble_new_font_alias:nnn {mc} {#1} {\mathcal}
  }
\clist_map_inline:nn {a,b,c,d,e,f,g,h,i,j,k,l,m,n,o,p,q,r,s,t,u,v,w,x,y,z}
  { \preamble_new_font_alias:nnn {mf} {#1} {\mathfrak} }
\ExplSyntaxOff

\catcode`,\active

\catcode`\,12

\newtheorem{theorem}{Theorem}[section]

\crefname{theorem}{Theorem}{Theorems}
\Crefname{theorem}{Theorem}{Theorems}
\crefname{thm}{Theorem}{Theorems}
\Crefname{thm}{Theorem}{Theorems}

\preamblesharedtheorem{proposition}{Proposition}
\crefname{proposition}{Proposition}{Propositions}
\Crefname{proposition}{Proposition}{Propositions}

\preamblesharedtheorem{prop}{Proposition}
\crefname{prop}{Proposition}{Propositions}
\Crefname{prop}{Proposition}{Propositions}

\preamblesharedtheorem{lemma}{Lemma}
\crefname{lemma}{Lemma}{Lemmas}
\Crefname{lemma}{Lemma}{Lemmas}

\preamblesharedtheorem{lem}{Lemma}
\crefname{lem}{Lemma}{Lemmas}
\Crefname{lem}{Lemma}{Lemmas}

\newtheorem*{lemma*}{Lemma}

\preamblesharedtheorem{claim}{Claim}
\crefname{claim}{Claim}{Claims}
\Crefname{claim}{Claim}{Claims}

\preamblesharedtheorem{corollary}{Corollary}
\crefname{corollary}{Corollary}{Corollaries}
\Crefname{corollary}{Corollary}{Corollaries}

\preamblesharedtheorem{conjecture}{Conjecture}
\crefname{conjecture}{Conjecture}{Conjectures}
\Crefname{conjecture}{Conjecture}{Conjectures}

\newtheorem{innerimportantthm}{Theorem}

\crefname{innerimportantthm}{Theorem}{Theorems}
\Crefname{innerimportantthm}{Theorem}{Theorems}

\newtheorem{innerimportantcor}{Corollary}

\crefname{innerimportantcor}{Corollary}{Corollaries}
\Crefname{innerimportantcor}{Corollary}{Corollaries}

\theoremstyle{definition}
\preamblesharedtheorem{definition}{Definition}
\crefname{definition}{Definition}{Definitions}
\Crefname{definition}{Definition}{Definitions}

\preamblesharedtheorem{nondefinition}{Non-Definition}
\crefname{nondefinition}{Non-Definition}{Non-Definitions}
\Crefname{nondefinition}{Non-Definition}{Non-Definitions}

\preamblesharedtheorem{exercise}{Exercise}
\crefname{exercise}{Exercise}{Exercises}
\Crefname{exercise}{Exercise}{Exercises}

\preamblesharedtheorem{example}{Example}
\crefname{example}{Example}{Examples}
\Crefname{example}{Example}{Examples}

\preamblesharedtheorem{remark}{Remark}
\crefname{remark}{Remark}{Remarks}
\Crefname{remark}{Remark}{Remarks}

\preamblesharedtheorem{observation}{Observation}
\crefname{observation}{Observation}{Observations}
\Crefname{observation}{Observation}{Observations}

\preamblesharedtheorem{question}{Question}
\crefname{question}{Question}{Questions}
\Crefname{question}{Question}{Questions}

\preamblesharedtheorem{problem}{Problem}
\crefname{problem}{Problem}{Problems}
\Crefname{problem}{Problem}{Problems}

\numberwithin{equation}{section}

\usepackage[style=alphabetic]{biblatex}
\newcommand{\shull}[1]{\widehat{#1}}

\begin{document}

\pagestyle{plain}

\title{An entire function that violates quasiminimality}
\author{Spencer Dembner}
\date{October 6, 2026}

\begin{abstract}
    Zilber's celebrated quasiminimality conjecture states that a subset of $\bbC$ defined by polynomials and exponentials is either countable, or has countable complement. Koiran has asked whether the same could be true replacing $\exp$ with any unary entire function. We give a negative answer to Koiran's question, using results from the theory of holomorphic approximation.
\end{abstract}

\maketitle

\section{Introduction}

Zilber's quasiminimality conjecture is one of the more basic unresolved questions about the exponential function.

\begin{conjecture}[\cite{ZilberGeneralizedAnalyticSets}]\label{conj:Quasiminimality}
    Let $\bbC_{\exp}$ be the first-order structure on $\bbC$ generated by polynomials and the complex exponential function. If $X \subset \bbC$ is $\bbC_{\exp}$-definable with parameters, then $X$ is countable or cocountable.
\end{conjecture}

Since it was first posed, \cref{conj:Quasiminimality} has inspired many important developments in geometry and logic, including the Pila--Wilkie counting theorem; we refer the reader to \cite{WilkieQuasiminimality} for more on its history. By work of Bays and Kirby \cite{BaysKirbyQuasiminimality}, \cref{conj:Quasiminimality} would follow from the Exponential-Algebraic Closedness conjecture, an exponential version of Hilbert's Nullstellensatz which has seen some progress in recent years. Wilkie \cite{WilkieQuasiminimality} has also proposed a direct approach to \cref{conj:Quasiminimality} via o-minimal structures.

Our goal is to understand what happens if $\exp \colon \bbC \to \bbC$ is replaced by another unary entire function $f$. A priori it is not obvious that $\exp$ should be special, and Koiran has even asked (see \cite{KirbyAroundQuasiminimality, WilkieQuasiminimality}) whether the structure generated by \emph{all} unary entire functions might be quasiminimal. Dmitrieva \cite{DmitrievaThesis} has shown there do exist ``Liouville-generic'' transcendental entire functions $f$ for which quasiminimality is true. However, our main theorem gives a negative answer to Koiran's question.

\begin{theorem}\label{thm:Counterexample}
    There exists a unary entire function $f \colon \bbC \to \bbC$ for which the structure
    \[ \bbC_f = (\bbC; 0,1,+, \times, f) \]
    is not quasiminimal.
\end{theorem}

To prove \cref{thm:Counterexample}, we first observe that a single entire function can encode finitely many others, for example by repeatedly taking even and odd parts.

\begin{proposition}\label{prop:ManyToOne}
    Let $f_1, \ldots, f_k$ be finitely many unary entire functions. Then there exists a single structure $\bbC_f$ in which all the $f_i$ are definable.
\end{proposition}

\begin{proof}
    First suppose that $k = 2$, and define
    \[ f(z) = f_1(z^2) + z f_2(z^2).  \]
    We have
    \[ f(z) + f(-z) = 2 f_1(z^2), \qquad f(z) - f(-z) = 2z f_2(z^2), \]
    and it follows that $f_1, f_2$ are $\bbC_f$-definable. Iterating, this allows us to encode any finite set of entire functions. \qedhere
\end{proof}

Let $A, B \colon \bbC \to \bbC^2$ be a pair of holomorphic curves, and let
\[ X = \{z \in \bbC: A(\bbC) \cap ((z,0) + B(\bbC)) \neq \varnothing \} \]
be the set of first-coordinate translates for which the images intersect. We will show that if $A$ and $B$ are constructed appropriately, the set $X$ and its complement are uncountable. More precisely, we will show that $X$ contains a nonempty open set, and that its complement contains a Cantor set. By \cref{prop:ManyToOne}, there is a structure $\bbC_f$ that defines the coordinates of $A$ and $B$, so this suffices.

To build $A$ and $B$, we start with a pair of maps that meet transversely at $0$, such as
\[ A_0(z) = (z,0), \qquad B_0(z) = (0,z). \]
This guarantees that $X$ will contain a neighborhood of zero, since small translations preserve a transverse intersection. The sets $A_0(\bbC), B_0(\bbC)$ are never disjoint under any translation, so for these functions we have $X = \bbC$. On the other hand, if we restrict to the closed unit disk, then there are many disjoint translations: For example, we have
\[ A_0(\ovl D(0,1)) \cap ((2,0) + B_0(\ovl D(0,1))) = \emptyset. \]
It turns out that we can extend this disjointness to a larger disk, as long as we allow an arbitrarily small perturbation on the unit disk. Iterating and passing to the limit, we can ensure the translates $A(\bbC), (2,0) +B(\bbC)$ are disjoint while preserving a transverse intersection between $A(\bbC)$ and $B(\bbC)$. If we instead replace the point $(2,0)$ by a branching tree of points (see \cref{sec:CompleteProof}), then the complement of $X$ acquires a Cantor set.

It is far from obvious that the perturbations required to make this argument exist. Their existence is part of the theory of Oka manifolds, and as we will explain in the next section, can be deduced from a theorem of Kusakabe.

\subsection*{Acknowledgments} I thank Ravi Vakil, Ben Bakker, James Freitag, and Gal Binyamini for interesting discussions on related questions. Alex Wilkie and Franc Forstnerič provided several useful comments on an earlier draft of this article.

\subsection*{AI Declaration} The first version of this result was found by GPT-6 Astra, working mostly autonomously except for some feedback on proposed strategies. However, this version relied on Theorem 1 of \cite{ForstnericRitterOkaBallComplements}, and its published proof is incomplete.\footnote{A proposed fix is described in \cite[Remark~1.3]{ForstnericProperImmersions}} The proof we give here is a modified version that avoids this dependence, found with both human and AI input. This paper was written entirely by the author.

\section{Polynomial convexity and polynomial approximation}

\begin{definition}
    The \emph{polynomial hull} of a compact subset $K \subset \bbC^n$ is
    \[ \shull{K} = \{ x \in \bbC^n : |P(x)| \leq \norm{P}_K \text{ for all polynomials } P \}, \]
    and $K$ is \emph{polynomially convex} if $\shull{K} = K$. The hull $\shull{K}$ is compact, and it is the smallest polynomially convex set containing $K$.
\end{definition}

Throughout the paper, we let
\[ \ovl D_r = \ovl D(0, r) \subset \bbC \]
denote the closed disk of radius $r$. The central ingredient in our proof is the following approximation result, which follows from work of Kusakabe.

\begin{theorem}\label{thm:KusakabeConsequence}
    If $L \subset \bbC^2$ is compact and polynomially convex, then every holomorphic map
    \[ f \colon U \to \bbC^2 \setminus L \]
    on a neighborhood of $\ovl D_r$ can be approximated uniformly on $\ovl D_r$ by entire maps
    \[ \bbC \to \bbC^2 \setminus L. \]
\end{theorem}

\begin{proof}
    See Definition 1.1 and Corollary 1.3 of \cite{KusakabeOkaComplements}. We are using only the one-dimensional case of the convex approximation property stated there. \qedhere
\end{proof}

\begin{corollary}\label{cor:KusakabeForUs}
    Let $L \subset \bbC^2$ be compact and polynomially convex, and let
    \[ f \colon U \to \bbC^2 \setminus L \]
    be a holomorphic map on a neighborhood of $\ovl D_r$ whose image avoids $L$. If $R > r$ and $\varepsilon > 0$, then there exists a polynomial map $F$ with
    \[ \norm{F - f}_r < \varepsilon, \qquad F(\ovl D_R) \cap L = \emptyset.  \]
\end{corollary}

\begin{proof}
    Applying \cref{thm:KusakabeConsequence}, let $E$ be an entire map with
    \[ \norm{E - f}_r < \varepsilon / 2, \qquad E(\bbC) \cap L = \emptyset. \]
    Since $E(\ovl D_R)$ is compact, the distance $d(E(\ovl D_R), L)$ is positive. If we let $F$ be a polynomial closely approximating $E$ on the compact set $\ovl D_R$, then the result follows. \qedhere
\end{proof}

In order to use this result we require a criterion for polynomial convexity, and this is what we discuss next. For brevity, let us call a polynomial map $\bbC \to \bbC^2$ \emph{birational} if it is birational onto its image, or equivalently is generically injective. It will be helpful later to recall that birational maps are dense among polynomial maps on $\ovl D_r$. Indeed, given any map $f = (f_1, f_2)$, we can take the map
\[ (f_1+\varepsilon z^n, f_2 + \varepsilon z^{n+1}) \]
for $n > \deg(f)$ and $\varepsilon \ll 1$. The former condition ensures that the map will be birational onto its image, since the degree of the corresponding field extension will divide both $n$ and $n+1$. The latter ensures that values on $\ovl D_r$ do not change much.

\begin{lemma}\label{lem:PolyConvexity}
    If
    \[ f_1, \ldots, f_k  \colon \bbC \to \bbC^2 \]
    are birational maps with pairwise distinct images and $r_1, \ldots, r_k$ are positive, then the set
    \[ K:= \bigcup_i f_i (\ovl D_{r_i}) \]
    is polynomially convex.
\end{lemma}

\begin{proof}
    Let
    \[ f \colon \coprod_i \bbC_i \to \bbC^2 \]
    be the coproduct of the maps $f_i$. A nonconstant polynomial map $\bbC \to \bbC^2$ is always proper because its leading term dominates, and it follows that $f$ is proper as well. By \cite[Theorem~3.1]{MihalacheRungeSteinSpaces}, it follows that $K$ is polynomially convex if and only if $f^{-1}(K)$ is. Since the $f_i$ are generically injective and their images only have finite overlap, we have
    \[ f^{-1}(K) = \coprod_i \ovl{D}_{r_i} \cup F, \]
    where $F$ is a finite set. If $z \in \bbC_i \setminus f^{-1}(K)$, we can separate it from $f^{-1}(K)$ using a function
    \[ \prod_{a \in F \cap \bbC_i} (z - a) z^n, \]
    as long as $n$ is sufficiently large. \qedhere
\end{proof}

Note that the pairwise distinctness hypothesis will always be satisfied for us. It only fails for a linear curve, translated along its tangent direction.

The consequence we will use in the construction is the following.

\begin{lemma}\label{lem:DoubleExtension}
    Let $p, q \colon \bbC \to \bbC^2$ be birational polynomial maps and let $R > r > 0$. Let $E \subset \bbC^2$ be finite; suppose that
    \[ p(\ovl D_r) \cap (q(\ovl D_r) + E) = \emptyset,  \]
    and that the translates of $q(\bbC)$ by $E$ are pairwise distinct. Then for all $\varepsilon > 0$, there exist birational polynomial maps $P, Q$ with
    \[ \norm{P - p}_r < \varepsilon, \qquad \norm{Q - q}_r < \varepsilon, \qquad P(\ovl{D}_R) \cap (Q(\ovl{D}_R)+ E) = \varnothing. \]
\end{lemma}

\begin{proof}
    By \cref{lem:PolyConvexity}, the set $L := (q(\ovl D_r) + E)$ is polynomially convex. Applying \cref{cor:KusakabeForUs} once, we obtain a map $P$ satisfying $\norm{P - p}_r < \varepsilon$, and
    \[ P(\ovl D_R) \cap (q(\ovl D_r) + E) = \emptyset. \]
    Since birational maps are dense, we can also assume that $P$ is birational, and nonlinear. The set $L = P(\ovl D_R) - E$ is likewise polynomially convex. Applying \cref{cor:KusakabeForUs} once again, we obtain the desired birational map $Q$. \qedhere
\end{proof}

\section{Completing the proof}\label{sec:CompleteProof}

By induction, we build a sequence $(A_i, B_i, E_i)$, where $A_i, B_i \colon \bbC \to \bbC^2$ are birational polynomial maps, and $E_i \subset \bbC$ is a finite set of size $2^i$. Later we will take the curves $A, B$ to be the limits of $(A_i)$ and $(B_i)$, and conclude as in the introduction. For our base case, we take
\[ A_0(z) = (z, 0), \qquad B_0(z) = (0, z), \qquad E_0 = \{2 \}. \]
Let $\vec e = (1,0)$. At each stage $i \geq 0$ we require that
\begin{equation}
    A_i(\ovl D_{2^i}) \cap (E_i \cdot \vec e + B_i(\ovl D_{2^i})) = \emptyset. \label{eq:inductivehypo}
\end{equation}

Applying \cref{lem:DoubleExtension}, we obtain birational maps $A_{i+1}, B_{i+1}$ which satisfy \eqref{eq:inductivehypo} with expanded radius $2^{i+1}$. Since disjoint compact sets have positive separation, there is some $\rho_i > 0$ such that
\[ A_{i+1}(\ovl D_{2^{i+1}}) \cap \left( (\ovl D_{\rho_i} + E_i) \cdot \vec e + B_{i+1}(\ovl D_{2^{i+1}}) \right) = \emptyset. \]
Let us write
\[ K_i = \ovl{D}_{\rho_i} + E_i. \]
Choose $\rho_i$ small enough that $K_i$ is a disjoint union of $2^i$ closed discs, and that $K_i \subset K_{i-1}$ for $i \geq 1$. Let $E_{i+1}$ be a collection of $2^{i+1}$ distinct points from $K_i^{\circ}$, two from each component. Then the inductive hypothesis is satisfied and the induction continues. The set
\[ K := \bigcap_i K_i \]
is compact and uncountable, because its points correspond to paths in an infinite binary tree.

At each stage in the induction, we are forced to allow an error $\varepsilon_i$ such that
\[ \norm{A_{i+1} - A_i}_{2^i} < \varepsilon_i, \qquad \norm{B_{i+1} - B_i}_{2^i} < \varepsilon_i. \]
We can choose $\varepsilon_i$ arbitrarily small, and the final step is to optimize these errors. First, we require that
\[ \sum_{k=0}^{\infty} \varepsilon_k \]
be finite, which implies that $(A_i), (B_i)$ converge locally uniformly to holomorphic maps
\[ A, B \colon \bbC \to \bbC^2. \]
Bounding the same sum, we can force $A(0), B(0)$ to be as close to $0$ as desired. By the Cauchy estimates, a bound also forces their derivatives on $\ovl D_{1/2}$ to be uniformly close to $(1,0)$ and $(0,1)$ respectively. It follows that $A(\bbC), B(\bbC)$ possess a transverse intersection near zero. Finally, let us define the distances
\[ \delta_i = d(A_{i+1}(\ovl D_{2^{i+1}}), K_i \cdot \vec e + B_{i+1}(\ovl D_{2^{i+1}})) > 0, \]
which depend only on $\varepsilon_0, \ldots, \varepsilon_i$. We require, for all $i \geq 0$, that
\[  \sum_{k = i+1}^{\infty} \varepsilon_k < \delta_i / 2. \]
The left-hand side bounds the total error $\norm{A_{i+1} - A}_{2^{i+1}}$ on $\ovl D_{2^{i+1}}$, and likewise for $B$. Thus, we have
\[ A(\ovl D_{2^{i+1}}) \cap (K_i \cdot \vec e +B(\ovl D_{2^{i+1}})) = \emptyset. \]
If $z \in K$, then we can let $i \to \infty$ and conclude that
\[ A(\bbC) \cap (z \cdot \vec e + B(\bbC)) = \emptyset. \]
Finally, as in the introduction let us define
\[ X = \{z \in \bbC: A(\bbC) \cap (z \cdot \vec e + B(\bbC)) \neq \emptyset \}. \]
The set $X$ contains a nonempty neighborhood of zero since $A(\bbC)$, $B(\bbC)$ meet transversely. By construction, it omits the uncountable set $K$. Thus, the proof of \cref{thm:Counterexample} is complete. \qed

\printbibliography

\end{document}